\documentclass{article}

\usepackage{amsmath,amsthm,amssymb,amsfonts}
\usepackage{graphicx,color}

\newcommand{\R}{\mathbb{R}}
\newcommand{\RR}{\mathbb{R}}
\newcommand{\N}{\mathbb{N}}
\newcommand{\un}{\mathbf{1}}
\newcommand{\eps}{\varepsilon}

\newcommand{\dis}{\mathcal{D}}

\newcommand{\ha}{H^{\frac12} (\Omega) }

\newcommand{\intom}{\int_\Omega}
\newcommand{\vphi}{\varphi}
\newcommand{\na}{\nabla}
\newcommand{\pa}{\partial}

\def\F{\mathcal F}
\def\E{\mathcal E}

\newtheorem{proposition}{Proposition}
\newtheorem{lemma}{Lemma}
\newtheorem{corollary}{Corollary}
\newtheorem{theorem}{Theorem}

\theoremstyle{definition}

\theoremstyle{remark}
\newtheorem{remark}{Remark}

\author{C. Imbert\footnote{CEREMADE, Universit\'e Paris-Dauphine, UMR
    CNRS 7534, place de Lattre de Tassigny, 75775 Paris cedex 16,
    France}~ and A. Mellet\footnote{Department of
    Mathematics. Mathematics Building. University of Maryland. College
    Park, MD 20742-4015, USAä}}

\title{Electrified thin films: Global existence of non-negative solutions}

\begin{document}

\maketitle

\begin{abstract} 
  We consider an equation modeling the evolution of a viscous liquid
  thin film wetting a horizontal solid substrate destabilized by an
  electric field normal to the substrate.  The effects of the electric
  field are modeled by a lower order non-local term.  We introduce the
  good functional analysis framework to study this equation on a
  bounded domain and prove the existence of weak solutions defined
  globally in time for general initial data (with finite energy).
\end{abstract}



\section{Introduction}

In this paper, we construct solutions for a thin film type equation
with a destabilizing singular integral term. This term models the
effects of an electric field (see \cite{tp07}). From the analytical
point of view, this paper belongs to the large body of literature
devoted to the thin film equation with destabilizing terms such as
long-wave unstable thin film problems \cite{bp98,bp00,bbw} or the
Kuramato-Sivashinsky equation in combustion and solidification
\cite{fs88,frankel}.

More precisely, we are considering the following equation, which is
{introduced} by Tseluiko and Papageorgiou in
\cite{tp07} (see also \cite{pp05}):
\begin{equation}\label{eq:pp}
\begin{array}{l}
u_t+\big(u^3( c u_{xx} -\alpha u- \lambda I(u))_x\big)_x = 0\qquad x\in[0,L], \quad t>0
\end{array}
\end{equation} 
(in \cite{tp07}, (\ref{eq:pp}) is supplemented with periodic boundary
conditions).  This equation models the evolution of a liquid thin film
(of height $u$) wetting a horizontal solid substrate which is subject
to a gravity field and an electric field normal to the substrate.  The
term $\lambda I(u) $ models the effects of the electric field on the
thin film. The operator $I(u)$ is a nonlocal elliptic operator of
order $1$ which will be defined precisely later on (for now, we can
think of it as being the half-Laplace operator:
$I(u)=-(-\Delta)^{1/2}u$).  When $\lambda>0$, it has a destabilizing
effect (it has the "wrong" sign).  The term $\alpha u$ accounts for
the effects of gravity, and it is also destabilizing when $\alpha <0$
("hanging film").  In \cite{tp07}, it is proved that despite these
destabilizing terms, positive smooth solutions of (\ref{eq:pp}) do not
blow up and remain bounded in $H^1$ for all time.

As mentioned above, there are many papers devoted to the study of thin film
equations with destabilizing terms. In its simplest form,  the thin film
equation reads:
\begin{equation}\label{eq:tf}
u_t + (f(u) u_{xxx})_x = 0,
\end{equation}
The existence of non-negative weak solutions for (\ref{eq:tf}) was
first established by F.~Bernis and A.~Friedman \cite{bf90} for
$f(u)=u^n$, $n> 1$.  Further results (existence for $n>0$ and further
regularity results) were later obtained, by similar technics, in
particular by E. Beretta, M. Bertsch and R. Dal Passo \cite{BBD} and
A. Bertozzi and M. Pugh \cite{BP1,BP2}.  Results in higher dimension
were obtained in particular by Gr\"un in \cite{Grun01,Grun95,DGG98}.
The thin film equation with lower order destabilizing terms has also
received a lot of interest. In particular in \cite{bp98,bp00}, the
following equation
\begin{equation}\label{eq:long-wave}
u_t + (f(u) u_{xxx} - g(u) u_{x})_x =0
\end{equation}
is considered.  Such a destabilizing term (which, unlike that of
(\ref{eq:pp}), is a local term of order $2$) models, for instance, the
effects of gravity for a hanging thin film, or van der Waals type
interactions with the solid substrate.  In \cite{bp98}, the
nonlinearities $f(u)=u^n$ and $g(u)=u^m$ are considered and it is
proved (among other things) that there is no blow-up for $m < n
+2$. In \cite{bp00}, for $f(u)=u$ and $g(u)= u^m$, it is proved that
there is blow-up for $m \ge 3$ and initial data in $H^1(\R)$ with
negative ``energy''. The reader is referred to \cite{bp00} for a
precise statement.

In our equation (\ref{eq:pp}), the nonlinearities in front of the
stabilizing and destabilizing terms are the same ($f(u)=g(u)=u^3$),
but the destabilizing term is elliptic of order $3$ and is nonlocal in
space.  It is known (see \cite{tp07}) that positive smooth solutions
of (\ref{eq:pp}) do not blow up. The goal of the present paper is to
prove the existence of global in time weak solutions for
(\ref{eq:pp}).  
\vspace{7pt}

Note that besides the existence of solutions, many important
properties of the thin film equation (\ref{eq:tf}) have been
investigated (finite speed of propagation of the support, waiting time
phenomenon, existence of source-type solutions \textit{etc.}).  The
key tools in many of these studies are various delicate integral
inequalities (in particular the so-called $\alpha$-entropy
inequalities and local entropy and energy inequalities. See
\cite{bf90,bbp95,BP2,bp98}).  It is not clear that similar functional
inequalities holds for (\ref{eq:pp}).  One reason is that the algebra
involving the operator $I(u)$ is considerably more difficult than that
of the Laplace operator. Another reason, is the obvious difficulty in
deriving local estimates (due to the nonlocal nature of the operator
$I(u)$). For that reason, we only address the existence issue in this
paper.  
\vspace{7pt}

As in \cite{bp98,bp00}, the main difficulty in 
  proving the existence of solutions for (\ref{eq:pp}) comes from the
fact that the energy (see (\ref{eq:energy00}) below) can take negative
values.  In order to obtain $H^1$ a priori estimate, one thus has to
use the conservation of mass which, for non-negative solutions, gives
a global in time $L^1$ bound for the solution (see Lemma
\ref{lem:h1}).

With such an estimate in hand, the existence of global in time
solutions should follow from the construction of approximated
solutions satisfying the right functional inequalities.  Typically,
one needs to regularize the mobility coefficient $u^3$. One way to
proceed is to replace the coefficient $u^3$ with $u^3+\eps$ so that
the equation becomes strictly parabolic. However, for such a
regularized equation one cannot show the existence of non-negative
solutions (the maximum principle does not hold for fourth order
parabolic equations) and Lemma~\ref{lem:h1} is of no use.  An
alternative regularization is to replace the mobility coefficient
$u^3$ with a function $f_\eps (u)$ which satisfies in particular
$f_\eps (u) \sim u^4/\eps$. For such (more degenerate) mobility
coefficient, solutions are expected to be strictly positive and
therefore smooth.  This second regularization procedure was first
suggested by Bernis and Friedman \cite{bf90} and is used in particular
by Bertozzi and Pugh \cite{BP1,bp98}. However, the local in time
existence for such a degenerate equation is not clear to us, since the
corresponding proofs in \cite{bf90} rely on Schauder estimates which
are not classical (and perhaps tedious) with our non-local singular
term $I(u)$. For this reason, we choose the first regularization
approach; this requires us to pay attention to the lack of positivity
of the approximated solutions.  In particular, the $L^1$ norm is not
controlled, and the $H^1$ norm will be controlled by combining the
energy inequality with the entropy inequality.  The idea of combining
the energy together with the entropy in order to get a Lyapunov
functional appeared previously in \cite{st04} where a thin film
equation with a nonlinear drift term is thouroughly studied in all
dimensions.

The main contribution of this paper is thus to introduce the precise
functional analysis framework to be used to treat the term $I(u)$ and
to provide a method for constructing weak solutions satisfying the
proper a priori estimates.
\vspace{5pt}

Rather than working in the periodic setting, we will consider equation
(\ref{eq:pp}) on a bounded domain with Neumann boundary conditions
(these Neumann conditions can be interpreted as the usual contact
angle conditions and seem physically more relevant - the periodic
framework could be treated as well with minor modifications).  Further
details about the derivation of (\ref{eq:pp}) will be given in Section
\ref{sec:phys}.  Since the gravity term is of lower order than the
electric field term, it is of limited interest in the mathematical
theory developed in this paper.  We will thus take
$$ 
\alpha =0 \quad \mbox{ and } \quad c=\lambda =1.
$$
We thus consider the following problem:
\begin{equation}\label{eq:0}
\left\{
\begin{array}{ll}
u_t+\big (f(u) ( u_{xx}- I(u))_x\big)_x = 0 & \mbox{ for } x\in\Omega,\quad t>0\\[3pt]
u_x=0 ,\; f(u) (u_{xx}-I(u))_x=0  & \mbox{ for } x\in\pa\Omega,\quad t>0\\[3pt]
u(x,0)=u_0(x)& \mbox{ for } x\in\Omega.
\end{array}
\right.
\end{equation} 
The domain $\Omega$ is a bounded interval in $\RR$; in the sequel, we will
always take $\Omega=(0,1)$.  The mobility coefficient $f(u)$ is a $\mathcal
C^1$ function $f:[0,+\infty)\rightarrow (0,+\infty)$ satisfying
\begin{equation}\label{eq:f}
f(u)\sim u^n \quad \mbox{ as $u\to 0$}
\end{equation}
for some $n >1$. The operator $I$ is a non-local elliptic operator of
order $1$ which will be defined precisely in Section \ref{sec:pre} as
the square root of the Laplace operator with Neumann boundary
conditions (we have to be very careful with the definition of $I$ in a
bounded domain).

\paragraph{A priori estimates.}
As for the thin film equation (\ref{eq:tf}), we prove the existence of
solutions for (\ref{eq:0}) using a regularization/stability argument.
The main tools are integral inequalities which provide the necessary
compactness.  Besides the conservation of mass, we will see that the
solution $u$ of \eqref{eq:0} satisfies two important integral
inequalities: We define the energy $\mathcal{E} (u)$ and the entropy
$e(u)$ by
$$
\mathcal{E} (u) (t) = \frac12 \int_\Omega (u_x^2(t) + u(t) I u(t)) dx
\quad \text{and} \quad e (u) (t) = \frac12\int_\Omega G(u(t)) dx
$$
where $G$ is a non-negative convex function such that $f G''=1$. 
Classical solutions of (\ref{eq:0}) then satisfy:
\begin{eqnarray}
\label{eq:nrj-intro}
\mathcal{E}(u) (t) + \int_0^t\int_\Omega f(u) \big[(u_{xx} - I(u))_x \big]^2 \,dx\,  ds
\le \mathcal{E} (u_0), \\
\label{eq:ent-intro}
e(u) (t) +  \int_0^t \int (u_{xx})^2 \, dx\, ds + \int_0^t \int u_x  I
(u)_x \,dx \, ds\le e(u_0). 
\end{eqnarray}
Similar inequalities hold for the thin film equation (\ref{eq:tf}).
However, we see here the destabilizing effect of the nonlocal term $I(u)$:
First, we note that as in \cite{bp98,bp00} the energy $\E(u)$ can be written as the difference of
two non-negative quantities:
\begin{equation}\label{eq:energy00}
\mathcal{E} (u) (t) = \|u(t)\|^2_{\dot{H}^1(\Omega)}-\|u (t)\|^2_{\dot{H}^{\frac12}(\Omega)},
\end{equation}
and may thus take negative values.  Similarly, the entropy dissipation can
be written as
$$ 
\int_0^t \int (u_{xx})^2 \, dx\, ds + \int_0^t \int u_x  I (u)_x \,dx \,ds 
=  \|u (t)\|^2_{\dot{H}^2(\Omega)} - \|u(t)\|^2_{\dot{H}^{\frac32}_N(\Omega)},
$$
so the entropy may not be decreasing.

Nevertheless, it is reasonable to expect (\ref{eq:0}) to have
solutions that exist for all times.  Indeed, as shown in \cite{tp07},
the conservation of mass, the inequality (\ref{eq:nrj-intro}) and the
following functional inequality (see Lemma~\ref{lem:h1})
$$
\| u \|^2_{\dot{H}^1 (\Omega)} \le \alpha \mathcal{E} (u) + \beta
\|u\|^2_{L^1 (\Omega)}, \quad \forall u\in H^1(\Omega),
$$
implies that non-negative solution remains bounded in
$L^\infty(0,T;H^1(\Omega))$ for all time $T$.

Furthermore, the interpolation inequality
\begin{eqnarray*}
  \|u(t)\|^2_{\dot{H}^{\frac32}_N} &\le& C  \|u(t)\|_{\dot{H}^{1}} \|u(t)\|_{\dot{H}^{2}} \\
  & \le & \frac12 \|u(t)\|^2_{\dot{H}^{2}} + \frac C2 \|u(t)\|^2_{\dot{H}^{1}}
\end{eqnarray*}
yields
$$
e (u) (t) + \frac12 \int_0^t \|u (r)\|^2_{\dot{H}^2}\, dr \le 
e(u_0) + \frac12 \int_0^t \|u(r)\|^2_{\dot{H}^{1}} \, dr,
$$
and so the entropy remains bounded for all time as well.

\paragraph{Main results.}
We now state the two main results proved in this paper. They should be
compared with Theorems~3.1 and 4.2 in \cite[pp.185\&194]{bf90}.  The first
one deals with non-negative initial data whose entropies are
finite.

We recall that $G$ is a non-negative convex function such that 
$$G''(u)=\frac{1}{f(u)} \quad \mbox{ for all $u>0$}.$$

% -------------------------------------------------------------------------------------
\begin{theorem}\label{thm:main}
Let $n>1$ and $u_0 \in H^1 (\Omega)$ be such that $u_0 \ge 0$ and 
\begin{equation}\label{eq:entropy_init}
 \int _\Omega G(u_0)\, dx <\infty.
 \end{equation}
For all $T>0$ there exists a function  $u(t,x)\geq 0$ with
$$
u \in \mathcal C (0,T;L^2 (\Omega)) \cap L^\infty
(0,T; H^1 (\Omega)), \quad u_x \in L^2 (0,T;H^1_0(\Omega))
$$
such that, for all
$\phi \in \dis ([0,T) \times \bar \Omega)$ satisfying $\phi_x = 0 $ on
$(0,T) \times \partial \Omega$,
\begin{multline}\label{eq:weak}
  \iint_Q u  \phi_t  - f (u) [ u_{xx}-
  I(u)] \phi_{xx}  - f'(u)u_x (u_{xx}-I(u)) \phi_x  \, dt\, dx \\
+ \intom u_0 (x) \phi(0,x) \, dx=0.
\end{multline}
Moreover, the function $u$ satisfies for  every $t \in
[0,T]$, 
\begin{eqnarray}
 \intom u (t,x) \, dx & = &  \intom u_0 (x) \, dx,\nonumber \\
  \mathcal{E} (u(t)) + \int_0^t \intom f (u) \big[(u_{xx} -
I(u))_x\big]^2 \, ds\, dx & \le & \mathcal{E} (u_0), \label{eq:nrj}\\
 \intom G(u (t))\, dx+  \int_0^t \intom
(u_{xx})^2+u_x I(u)_x \, ds dx 
&\le& \intom G(u_0) \, dx .
\label{eq:ent}
\end{eqnarray}
\end{theorem}
%---------------

We point out that the weak formulation (\ref{eq:weak}) involve two
integrations by parts.  Our second main result is concerned with
non-negative initial data whose entropies are possibly infinite (this
is the case if $u_0$ vanishes on an open subset of $\Omega$ and $n\geq
2$).  In that case, only one integration by parts is possible, and the
solutions that we construct are weaker than those constructed in
Theorem \ref{thm:main}.  In particular, the equation is only satisfied
on the positivity set of the solution and the boundary conditions are
satisfied in a weaker sense.
%---------------
\begin{theorem}\label{thm:second}
Assume $n >1$ and let $u_0 \in H^1 (\Omega)$ be such that  $u_0 \ge 0$. 
For all $T>0$ there exists a function  $u(t,x)\geq 0$ such that
$$
u \in \mathcal C (0,T;L^2 (\Omega)) \cap
  L^\infty(0,T; H^1 (\Omega)) \cap \mathcal
  C^{\frac12,\frac18}(\Omega\times(0,T))
$$
such that 
\begin{eqnarray*}
  f (u) [ u_{xx}-  I(u)]_x \in L^2 (P)
\end{eqnarray*}
and such that, for all $\phi \in \dis ([0,T) \times \bar \Omega)$, 
\begin{equation}\label{eq:weak-bis}
  \iint_Q u  \phi_t \, dt\, dx + \iint_P f (u) [ u_{xx}-
  I(u)]_x \phi_x  \, dt\, dx 
+ \intom u_0 (x) \phi (0,x) \, dx= 0
\end{equation}
where $P = \{ (x,t) \in \bar Q: u(x,t)>0,t>0\}$.  Moreover, the function $u$
satisfies the conservation of mass and the energy inequality \eqref{eq:nrj}.

Finally, $u_x$ vanishes at all points $(x,t)$ of $\partial \Omega \times (0,T)$
such that $u(x,t) \neq 0$.
\end{theorem}
%---------------

\paragraph{Comments.}
These results are comparable to those of \cite{bf90} when $\lambda=0$.
The reader might be surprised that they are presented in a different
order than in \cite{bf90}.  The reason has to do with the proofs;
indeed, in contrast with \cite{bf90}, weak solutions given by
Theorem~\ref{thm:second} are constructed as limits of the solutions
given by Theorem~\ref{thm:main}.  This is because the entropy is
needed in order to construct the non-negative solutions. See the
discussion at the beginning of Section~\ref{sec:proof1} for further
details.

As pointed out earlier, the non-linearities in
front of the stabilizing ($u_{xxx}$) and destabilizing ($(I(u))_x$)
are the same. This is in contrast with the work of Bertozzi and Pugh
\cite{bp98,bp00}. By analogy with \eqref{eq:long-wave}, one could
consider the equation
$$
u_t + (u^n u_{xxx} - u^m (I(u))_x)_x =0
$$
in which case a scaling analysis similar to that of \cite{bp98}
suggests that blow up can only occur if $m \ge n+1$. However, to our
knowledge, there is no physical motivation for such a generalization
(in our case).

\paragraph{Organization of the article.} In Section~\ref{sec:phys}, we give
more details about the physical model leading to (\ref{eq:0}).  We
gather, in Section~\ref{sec:pre}, material that will be used throughout the
paper. In particular, we detail the functional analysis framework and the
definition of the non-local operator $I$ (which is similar to that used in
\cite{im}). Section~\ref{sec:regularized},~\ref{sec:proof1} and
\ref{sec:proof2} are devoted to the proofs of the main results. Finally, we
give in Appendix a technical result which is more or less classical. 

\paragraph{Acknowledgements.} The first author was partially supported
by the French Ministry of Research (ANR project “EVOL”). The second
author was partially supported by NSF Grant DMS-0901340.

\section{Physical model}
\label{sec:phys}
In this section, we briefly recall the  derivation of
(\ref{eq:0}) (see \cite{tp07} for further details).  We consider a
viscous liquid film which completely wets a solid horizontal substrate
and is constrained between two solid walls (at $x=0$ and $x=1$), see
Figure~1.
\begin{figure}[h]
\begin{center}
\includegraphics[height=3cm]{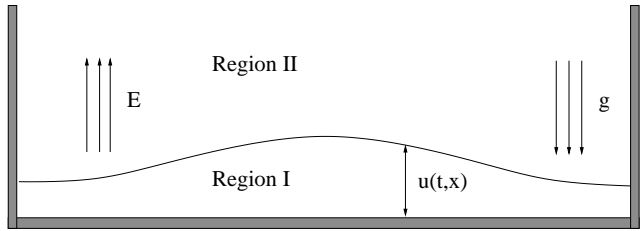}
\caption{A viscous thin film submitted to an electric field ${\bf E}$ and
  gravity $g$}
\end{center}\label{fig:film}
\end{figure}
The fluid is Newtonian and is assumed to be a perfect conductor.  The
substrate is a grounded electrode held at zero voltage.  Thanks to the
presence of another electrode (at infinity), an electric field ${\bf E}$ is
created which is constant at infinity (in the direction perpendicular
to the substrate):
$$ 
{\bf E}(x,y)\longrightarrow (0,E_0) \qquad \mbox{ as }
y\rightarrow +\infty.
$$

The height of the fluid is denoted by $u(t,x)$. Under the assumptions
of the lubrication approximation, it is classical that the evolution
of $u$ is described by Poiseuille's law:
\begin{equation}\label{eq:pois}
u_t -\pa_x \left(\frac{u^3}{3\mu}\pa_x p\right)=0
\end{equation}
where $p$ is the pressure at the free surface of the fluid $y=u(t,x)$. 
This pressure is the sum of three terms:
\begin{enumerate}
\item The capillary pressure due to surface tension, which can be approximated by 
$$ 
p_1 \sim -\sigma u_{xx}
$$
(replacing the mean curvature operator by the Laplacian).
\item The effect of gravity, given by
$$p_2 =   gu.$$
\item The additional pressure due to the action of the electric field $E$.
\end{enumerate}
To compute the third term appearing in the pressure, we introduce the
potential $V$ such that ${\bf E}=-\na V$, which satisfies
$$ 
\Delta V=0 \mbox{ for } y\geq u(x)
$$
and 
$$ V(x,y)=0 \quad \mbox{ on $y=u(x)$.}$$
The condition at $y\rightarrow \infty$ means that we can write
$$ V\sim E_0(Y_0 - y)$$ 
with (using standard linear approximation)
\begin{equation}\label{eq:extt}
\left\{
\begin{array}{ll}
\Delta Y_0=0 & \mbox{ for } y>0 , \; x\in \Omega\\
\nabla Y_0 \rightarrow 0 & \mbox{ as } y\rightarrow \infty  , \; x\in \Omega\\
Y_0 (x,0)=  u(x)  & \mbox{ for }  x\in \Omega.
\end{array}
\right.\end{equation}
At the boundary of the cylinder, we assume that the electric field has no horizontal component:
$$ \pa _x V=0,\quad  \mbox{ for } x\in\pa\Omega , \; y>0.$$
The pressure exerted by the electric field is then proportional to 
$$ 
p_3 = \gamma {\bf E}_y = -\gamma \pa_y V(x,0)=-\gamma E_0(\pa_y Y_0-1).
$$
The application $u\mapsto \pa_y Y_0(x,0)$ is a Dirichlet-to-Neumann
map for the harmonic extension problem (\ref{eq:extt}).  We denote
this operator by $I(u)$. We will see in Section~\ref{sec:pre} that
$I(u)$ is in fact the square root of the Laplace operator on the
interval $\Omega$ with homogeneous Neumann boundary conditions.

We thus have
$$
p =p_1+p_2+p_3=  - \sigma u_{xx} + g u - \gamma E_0 I (u) +c_0
$$
for some constant $c_0$, and we obtain \eqref{eq:pp} with
$c=\frac{\sigma}{3\mu}$, $\alpha = \frac{g}{3\mu}$ and $\lambda =
-\frac{\gamma E_0}{3 \mu}$.  Note that Poiseuille's law (\ref{eq:pois}) is obtained
under the no-slip condition for the fluid along the solid
support. Other conditions, such as the Navier slip condition leads to
$$ 
f(u)=u^3+\Lambda u^s
$$
with $s=1$ or $s=2$. This explains the interest of the community for
general diffusion coefficient $f(u)$.

\paragraph{Boundary conditions.}
Along the boundary $\pa \Omega$, the fluid is in contact with a solid
wall. It is thus natural to consider a contact angle condition at
$x=0$ and $x=1$: Assuming that the contact angle is equal to $\pi/2$,
we then get the boundary condition
$$ u_x=0\qquad \mbox{ on } \pa\Omega.$$

In \cite{tp07}, the authors derive
their analytic results in a periodic setting which is obtained
by considering the even extension of $u$ to the interval $(-1,1)$ (recall
that $\Omega=(0,1)$) and then taking the periodic extension (with period
$2$) to $\R$.

Finally, since the equation is of order $4$, we need an additional boundary condition. 
We thus assume that $u$ satisfies the following null-flux condition 
$$
u^3 (u_{xx}-I(u))_x=0 \qquad \mbox{ on } \pa\Omega
$$
which will guarantee the conservation of mass.

\section{Preliminaries}\label{sec:pre}

In this section, we recall how the operator $I$ is defined (see
\cite{im}) and give the functional analysis results that we will need
to prove the main theorem.  A very similar operator, with Dirichlet
boundary conditions rather than Neumann boundary conditions, was
studied by Cabr\'e and Tan \cite{ct09}.

\subsection{Functional spaces}

\paragraph{The space $H^s_N (\Omega)$.}
We denote by $\{\lambda_k,\vphi_k\}_{k= 0,1,2\dots}$ the eigenvalues
and corresponding eigenfunctions of the Laplace operator in $\Omega$
with Neumann boundary conditions on $\pa\Omega$:
\begin{equation}\label{eq:eigen}
\left\{
\begin{array}{ll}
-\Delta \vphi_k = \lambda_k\vphi_k & \mbox{ in } \Omega\\
\pa_\nu \vphi_k = 0 & \mbox{ on } \pa\Omega,
\end{array}
\right. 
\end{equation}
normalized so that $\int_\Omega \vphi_k^2\, dx=1$.  When $\Omega =
(0,1)$, we have
$$
\lambda_0=0\, , \qquad \vphi_0(x) = 1
$$
and
$$ 
\lambda_k = (k\pi)^2 \, , \qquad \vphi_k(x) = \sqrt 2 \cos(k\pi x)\,
\qquad k= 1,\,2,\, 3, \, \dots
$$
The $\vphi_k$'s clearly form an orthonormal basis of $L^2(\Omega)$.
Furthermore, the $\vphi_k$'s also form an orthogonal basis of the
space $H^s_N(\Omega)$ defined by
$$
H^s_N(\Omega) = \left\{u=\sum_{k=0}^\infty c_k\vphi_k \, ;\,
  \sum_{k=0}^\infty c_k^2(1+\lambda_k^s) <+\infty \right\}
$$
equipped with the norm
$$ 
||u||_{H^s_N(\Omega)}^2= \sum_{k=0}^\infty c_k^2(1+\lambda_k^s) 
$$
or equivalently (noting that $c_0= \int_\Omega u(x)\, dx$ and
$\lambda_k\geq 1$ for $k\geq 1$):
$$ 
||u||_{H^s_N(\Omega)}^2= \|u\|_{L^1 (\Omega)}^2 + \| u\|^2_{\dot{H}^s_N (\Omega)}
$$
where the homogeneous norm is given by:
$$
\| u\|^2_{\dot{H}^s_N (\Omega)} = \sum_{k=1}^\infty c_k^2\lambda_k^s .
$$

\paragraph{A characterisation of $H^s_N (\Omega)$.}
The precise description of the space $H^s_N(\Omega)$ is a classical
problem.

Intuitively, for $s<3/2$, the boundary condition $u_\nu=0$ does not
make sense, and one can show that (see Agranovich and Amosov
\cite{aa03} and references therein):
$$ 
H^s_N(\Omega)=H^s(\Omega) \quad \mbox{ for all } 0\leq s<\frac32.
$$
In particular, we have $H^\frac{1}{2}_N(\Omega)=\ha$ and we will see
later that 
$$
\|u\|^2_{\dot{H}^{\frac12} (\Omega)} = \int_\Omega \int_\Omega (u(y)-u(x))^2 \nu (x,y) dx dy
$$
where $\nu (x,y)$ is a given positive function; see \eqref{defi:nu} below.

For $s>3/2$, the Neumann condition has to be taken into account, and
we have in particular
$$  
H^2_N(\Omega) = \{ u\in H^2(\Omega)\, ;\, u_\nu=0 \mbox{ on }\pa\Omega\}
$$
which will play a particular role in the sequel.  More generally, a
similar characterization holds for $3/2<s<7/2$. For $s>7/2$,
additional boundary conditions would have to be taken into account,
but we will not use such spaces in this paper.  In
Section~\ref{sec:regularized}, we will also work with the space $H^3_N
(\Omega)$ which is exactly the set of functions in $H^3(\Omega)$
satisfying $u_\nu = 0$ on $\partial \Omega$.

The case $s=3/2$ is critical (note that $u_\nu|_{\pa\Omega}$ is not
well defined in that space) and one can show that
\begin{eqnarray*} 
  H^{\frac32}_N(\Omega) 
  & = & \left\{ u\in H^{\frac32}(\Omega)\, ;\, \int_\Omega \frac{u_x^2}{d(x)}\, dx <\infty\right\}  
\end{eqnarray*}
where $d(x)$ denotes the distance to $\pa\Omega$. A similar result
appears in \cite{ct09}; more precisely, such a characterization of
$H^{\frac32}_N(\Omega)$ can be obtained by considering functions $u$
such that $u_x \in \mathcal V_0(\Omega)$ where $\mathcal V_0(\Omega)$
is defined in \cite{ct09} as the equivalent of our space
$H^{1/2}_N(\Omega)$ with Dirichlet rather than Neumann boundary
conditions.  We do not dwell on this issue since we will not need this
result in this paper.

\subsection{The operator $I$}

As it is explained in the Introduction, the operator $I$ is related to
the computation of the pressure as a function of the height of the
fluid. 

\paragraph{Spectral definition.} With $\lambda_k$ and
$\vphi_k$ defined by (\ref{eq:eigen}), we define the operator
\begin{equation}\label{eq:Idef}
  I: \sum_{k=0}^{\infty} c_k \vphi_k \; \longmapsto \; 
- \sum_{k=0}^{\infty} c_k \lambda_k^{\frac12} \vphi_k
\end{equation}
which clearly maps $H^1(\Omega)$ onto $L^2(\Omega)$ and
$H^2_N(\Omega)$ onto $H^1(\Omega)$.

\paragraph{Dirichlet-to-Neuman map.} 
We now check that this definition of the operator $I$ is the same as
the one given in Section \ref{sec:phys}, namely $I$ is the
Dirichlet-to-Neumann operator associated with the Laplace operator
supplemented with Neumann boundary conditions:

We consider the following extension problem:
\begin{equation}\label{eq:exp}
\left\{
\begin{array}{ll}
-\Delta v = 0 & \mbox{ in } \Omega \times (0,+\infty), \\
v(x,0)= u(x) &\mbox{ on } \Omega, \\
v_\nu = 0 & \mbox{ on } \pa \Omega \times (0,\infty).
\end{array}
\right.
\end{equation}
Then, we can show (see \cite{im}):
\begin{proposition}[\cite{im}]
  For all $u\in H^{\frac12}_N(\Omega)$, there exists a unique extension $v\in
  H^{1}( \Omega \times (0,+\infty))$ solution of (\ref{eq:exp}).  
\item Furthermore, if $u(x)=\sum_{k=1}^{\infty} c_k \vphi_k(x)$, then
\begin{equation}\label{eq:v}
v(x,y)=\sum_{k=1}^{\infty} c_k \vphi_k(x) \exp(-\lambda_k^{\frac12}y).
\end{equation}
\end{proposition}
%--------------------
and we have:
\begin{proposition}[\cite{im}] \label{prop:dir-to-neum} 
For all $u\in H^2_N(\Omega)$, we
  have
$$
I(u)(x)= - \frac{\pa v}{\pa\nu}(x,0) =\pa_y v(x,0)\quad \mbox{ for all $x \in \Omega$,}
$$
where $v$ is the unique harmonic extension solution of (\ref{eq:exp}).
\item Furthermore $I\circ I (u)= -\Delta u$.
\end{proposition}
%--------------------

\paragraph{Integral representation.}
Finally, the operator $I$ can also be represented as a singular integral operator:
%--------------------------
\begin{proposition}[\cite{im}]\label{prop:integral-rep}
Consider a smooth function $u: \Omega \to \R$. Then for all $x \in \Omega$,
$$
I(u)(x) = \int_\Omega (u(y) - u(x)) \nu (x,y) dy 
$$
where $\nu(x,y)$ is defined as follows: for all
$x,y \in \Omega$, 
\begin{equation}\label{defi:nu}
  \nu (x,y) =  \frac{\pi}2 \left( \frac{1}{1 - \cos (\pi (x-y))}
    + \frac{1}{1 - \cos (\pi (x+y))} \right).
\end{equation}
\end{proposition}
%-------------------------

\subsection{Functional equalities and inequalities}

\paragraph{Equalities.} The semi-norms
$||\cdot||_{\dot{H}^{\frac12}(\Omega)}$, $||\cdot||_{\dot{H}^{1}(\Omega)}$, $||\cdot
||_{\dot{H}^{\frac32}_N(\Omega)}$ and $||\cdot
||_{\dot{H}^{2}_N(\Omega)}$ are related to the operator $I$ by
equalities which will be used repeatedly.
%---------------------------------------------------------------
\begin{proposition}[The operator $I$ and several semi-norms -- \cite{im}]
\label{prop:semi-norms}
\item For all $u\in H^{\frac12}(\Omega)$, we have 
$$
-\int u I(u)\, dx = \frac12 \int_\Omega \int_\Omega (u(x)-u(y))^2 \nu (x,y) dx dy
 = ||u||^2_{\dot{H}^{\frac12}(\Omega)}
$$
where $\nu$ is defined in \eqref{defi:nu}. 
\item For all $u\in H^2_N(\Omega)$, we have
$$ 
- \int_\Omega u_x  I(u)_x\, dx = ||u||_{\dot{H}^{\frac32}_N(\Omega)}^2. 
$$
\item For all $k \in \N$ and $u\in H^{k+1}_N(\Omega)$, we have
\begin{equation}\label{eq:semi-norm}
 \int_\Omega  (\pa_x^k I (u))^2 \, dx  = \|u\|_{\dot{H}^{k+1}_N (\Omega)}^2.
\end{equation}
\end{proposition}
%----------------------------------------------------------------

\paragraph{Inequalities.}
First, we recall the following Nash inequality:
$$
\|u\|_{L^2 (\Omega)} \le C \|u\|^{\frac13}_{H^1 (\Omega)}
\|u\|^{\frac23}_{L^1 (\Omega)}.
$$
It implies in particular that,
\begin{equation}\label{eq:nash}
\|u\|^2_{L^2 (\Omega)} \le \frac{1}{2}\|u\|^2_{\dot{H}^1 (\Omega)}
+ C \|u\|^2_{L^1 (\Omega)}.
\end{equation}

This inequality will allow us to control the $H^1$ norm by the
energy $\mathcal{E}(u)$ and the $L^1$ norm. 
Indeed, we recall that the energy is defined by
\begin{equation}\label{def:energy}
  \mathcal{E} (u) = \frac1 2 \int_\Omega |u_x|^2 +u\,I(u)\, dx = 
  \frac12 \|u\|^2_{\dot{H}^1(\Omega)} - \frac12\|u\|^2_{\dot{H}^{\frac12} (\Omega)}.
\end{equation}
We then have:
%---------------------------
\begin{lemma}\label{lem:h1}
  There exist  positive constants $\alpha,\beta$ such that
  for all $u \in H^1( \Omega)$,
$$
\| u \|^2_{\dot{H}^1 (\Omega)} \le \alpha \mathcal{E} (u) + \beta \|u\|^2_{L^1 (\Omega)}.
$$
\end{lemma}
%---------------------------
\begin{remark}
See also Lemma~4.1 in \cite{tp07}. 
\end{remark}
\begin{proof}
We have 
$$  \|u\|^2_{\dot{H}^1(\Omega)}= 2 \mathcal{E} (u) + \|u\|^2_{\dot{H}^{\frac12} (\Omega)},$$
 and using \eqref{eq:semi-norm} with $k=0$
 and \eqref{eq:nash}, we get:
\begin{eqnarray*}
  \|u\|^2_{\dot{H}^{\frac12} (\Omega)} &=& -\int u \,I(u)\, dx \\
&\leq &  \|u\|_{L^2(\Omega)} \|I(u)\|_{L^2(\Omega)} \\
&\le&  \|u\|_{L^2(\Omega)} \|u\|_{\dot{H}^1(\Omega)}  \\
  &\le &  \frac12 \|u\|_{L^2(\Omega)}^2+\frac12 \|u\|^2_{\dot{H}^1(\Omega)}\\
 &\le& \frac34  \|u\|^2_{\dot{H}^1(\Omega)} + \frac{C}2 \|u\|^2_{L^1(\Omega)},
\end{eqnarray*}
hence the result.
\end{proof}

\section{A regularized problem}\label{sec:regularized}
We can now turn to the proof of Theorem \ref{thm:main} and
\ref{thm:second}.  As usual, we introduce the following regularized
equation:
\begin{equation}\label{eq:reg}
\left\{
\begin{array}{ll}
u_t+(f_\eps (u) (u_{xx}- I(u))_x)_x = 0 & \mbox{ for } x\in\Omega,\quad t>0\\[3pt]
u_x=0 ,\; f_\eps (u) (u_{xx} -I(u))_x=0  & \mbox{ for } x\in\pa\Omega,\quad t>0\\[3pt]
u(x,0)=u_0(x)& \mbox{ for } x\in\Omega
\end{array}
\right.
\end{equation} 
where the mobility coefficient $f(u)$ is approximated by
$$f_{\eps} (u)= \min (\max(\eps,f(|u|)),M)$$
which satisfies
$$ \eps \leq f_\eps(u)\leq M\quad \mbox{ for all $u\in\RR$} .$$
Ultimately, we will show that the solution $u$ satisfies
$$0\leq u(t,x)\leq M_0$$
for some constant $M_0$ independent of $M$, so that we do not have to
worry about $M$ (provided we take it large enough). The $\eps$ is of
course the most important parameter in the regularization since it
makes (\ref{eq:reg}) non-degenerate.

The proof of Theorem \ref{thm:main} consists of two parts: First, we
have to show that the regularized equation (\ref{eq:reg}) has a
solution (which may take negative values). Then we must pass to the
limit $\eps\to0$ and show that we obtain a non-negative solution of
\eqref{eq:0}.

In this section, we prove the first part. Namely, we prove (to be
compared with Theorem~1.1 in \cite{bf90}): 
\begin{theorem}\label{thm:existence-reg}
  Let $u_0 \in H^1 (\Omega)$. For all $T>0$ there exists a function
  $u^\eps(t,x)$ such that
$$
u^\eps \in \mathcal C(0,T;L^2(\Omega)) \cap L^\infty
(0,T; H^1 (\Omega)) \cap L^2 (0,T;H^3_N (\Omega))
$$
 such that, for all $\phi \in \dis ([0,T) \times \bar \Omega)$, 
\begin{equation}\label{eq:reg-weak}
\iint_Q u^\eps  \phi_t  + f_\eps (u^\eps) [ u^\eps_{xx}-
I(u^\eps)]_x \phi_x  \, dt\, dx + \intom u_0(x) \phi (0,x) \, dx= 0.
\end{equation}
Moreover, the function $u^\eps$ satisfies for  every $t \in
[0,T]$, 
\begin{eqnarray}
 \intom u^\eps (t,x) \, dx & = &  \intom u_0 (x) \, dx,\nonumber \\
  \mathcal{E} (u^\eps(t)) + \int_0^t \intom f_\eps (u^\eps) \big[(u^\eps_{xx} -
I(u^\eps))_x\big]^2 \, ds\, dx & \le & \mathcal{E} (u_0), \label{eq:nrj-eps}
\end{eqnarray}
and
\begin{multline}
\intom (u^\eps_x)^2 \, dx +  \int_0^t \intom
f_\eps(u^\eps)(u_{xxx}^\eps)^2\, ds\,  dx \\
\le \intom((u_0)_x)^2 \,dx + \int_0^t\int_\Omega f_\eps(u^\eps) u^\eps_{xxx}
(I(u^\eps))_x  \, ds\, dx
\label{eq:estim-h1}
\end{multline}
and 
\begin{equation}
  \intom G_\eps(u^\eps (t))\, dx+  \int_0^t \intom
  (u^\eps_{xx})^2+u^\eps_x I(u^\eps)_x \, ds dx 
  \le \intom G_\eps(u_0) \, dx.
\label{eq:ent-eps}
\end{equation}
where $G_\eps$ is a non-negative function such that $f_\eps
G_\eps'' =1$. 

Finally, $u^\eps$ is $\frac12$-H\"older continuous
  with respect to $x$ and $\frac18$-H\"older continuous with respect
  to $t$; more precisely, there exists a constant $C_0$ only depending
  on $\Omega$ and $\|u^\eps\|_{L^\infty (0,T;H^1 (\Omega))}$ and
  $\|f_\eps (u^\eps) [ u^\eps_{xx}- I(u^\eps)]_x \|_{L^2(Q)}$ such
  that
\begin{equation}\label{eq:holder}
\| u^\eps \|_{\mathcal C^{\frac12,\frac18}_{t,x} (Q)} \le C_0.
\end{equation}
\end{theorem}
%-----------------------------------------------------------------------------
\begin{remark}
This theorem is very similar to Theorem 1.1 in \cite{bf90}, and
some steps in our proof follow along the lines of
\cite{bf90}. For instance, getting the H\"older estimates from the 
$L^\infty H^1$ estimate is done in the same way. However, the main
difficulty in proving this theorem is precisely to get the $L^\infty
H^1$ estimate; this step is not straightforward at all and this is a
significant difference with \cite{bf90}. 
\end{remark}
\begin{proof}[Proof of Theorem~\ref{thm:existence-reg}]
Theorem~\ref{thm:existence-reg} follows from a fixed point argument: 
For some $T_*$, we denote
$$
V= L^2 (0,T_*;H^2_N (\Omega))
$$  
and we define the application $\mathcal F: V\rightarrow V$ such that
for $v\in V$, $\mathcal F(v)$ is the solution $u$ of
\begin{equation}\label{eq:fp}
\left\{
\begin{array}{ll}
 u_t+ (f_\eps(v) (u_{xxx}- I(v)_x))_x = 0 & \mbox{ for } x\in\Omega,\quad t>0\\[3pt]
u_x=0 ,\; f_\eps (v) (u_{xx} -I(v))_x=0  & \mbox{ for } x\in\pa\Omega,\quad t>0\\[3pt]
u(x,0)=u_0(x)& \mbox{ for } x\in\Omega.
\end{array}
\right.
\end{equation}

The fact that $\mathcal F$ is well defined 
follows from the observation that for $v\in V$, we have
$$
a(t,x) = f_\eps (v(t,x)) \in [\eps,M] \quad \text{ and } \quad 
g(t,x) = I (v)_x \in L^2 (Q)
$$
and the following proposition:
%--------------------
\begin{proposition} \label{prop:corner} 
Consider $u_0 \in H^1 (\Omega)$ and
  $a(t,x) \in L^\infty (Q)$ such that $\eps \le a (t,x) \le M$ 
  a.e. in $Q$. If $g \in L^2(Q)$, then there exists a 
  function
$$
u \in \mathcal{C}(0,T;L^2 (\Omega)) \cap 
L^\infty (0,T;H^1 (\Omega)) \cap L^2(0,T;H^3_N (\Omega))
$$
 and
$$
\iint_Q \big[ u \phi_t + a (u_{xxx}-g) \phi_x \big] \, dt \,dx + \intom u_0
(x) \phi (0,x) \, dx= 0
$$
for all $\phi \in \dis ([0,T)\times \bar \Omega)$. Moreover, for every $t \in
[0,T]$, $u$ satisfies
$$
\intom u(t,x) \, dx = \intom u_0 (x) \, dx \\
$$
and 
\begin{eqnarray}
  \intom (u_x)^2(t)\, dx + \frac12 \int_0^t
  \intom a (u_{xxx})^2 \, ds\,  dx \le \frac{M}2 \int_0^t \intom g^2
  \, ds\,  dx +\intom (u_0)_x^2 \, dx .  \label{eq:nrj-corner}
\end{eqnarray}
Furthermore, $u$ is $\frac12$-H\"older continuous with
  respect to $x$ and $\frac18$-H\"older continuous with respect to
  $t$; more precisely, there exists a constant $C_0$ only depending on
  $\Omega$ and $\|u\|_{L^\infty (0,T;H^1 (\Omega))}$ and
  $\|u_{xxx}-g\|_{L^2(Q)}$ such that
\begin{equation}\label{eq:holder-corner}
\| u \|_{\mathcal C^{\frac12,\frac18}_{t,x} (Q)} \le C_0.
\end{equation}
\end{proposition}
%--------------------

This proposition is a very natural existence result for the fourth order
linear parabolic equation
$$ 
u_t + (a u_{xxx})_x = (ag)_x.
$$
Its proof is fairly classical, we give some details in
Appendix~\ref{app:corner-proof} for the interested reader.
\vspace{20pt}

Next, we show the following result:
\begin{lemma}\label{lem:pf}
  There exists a (small) time $T_*>0$, depending only on $\eps$, $M$ and
  $\Omega$, such that $\F$ has a fixed point $u$ in $V= L^2
  (0,T_*;H^2_N (\Omega))$ for any initial data $u_0\in H^1(\Omega)$.
  Furthermore, $u$ satisfies
$$\|u \|_V \le R ||u_0||_{\dot{H}^1(\Omega)}$$
and 
\begin{equation}\label{eq:propa}
\|u \|_{L^\infty(0,T_*;\dot{H}^1(\Omega))} \leq  \sqrt{2} ||u_0||_{\dot{H}^1(\Omega)} .
\end{equation}
\end{lemma}
Before proving this lemma, let us complete the proof of
Theorem~\ref{thm:existence-reg}.

\paragraph{Construction of a solution for large times.}
Lemma \ref{lem:pf} gives the existence of a solution $u^\eps_1$ of
\eqref{eq:reg} defined for $t\in [0,T_*]$.  Since $T_*$ does not
depend on the initial condition, we can apply Lemma \ref{lem:pf} to
construct a solution $u^\eps_2$ in $[T_*,2T_*]$ with initial condition
$u^\eps_1(T_*,x)$ which is $H^1 (\Omega)$ by \eqref{eq:propa}. This
way, we obtain a solution $u^\eps$ of \eqref{eq:reg} on the time
interval $[0,2T_*]$.  Note that we also have
$$
\|u^\eps \|_{L^\infty(0,2T_*;\dot{H}^1(\Omega))} \leq \sqrt{2}^2
||u_0||_{\dot{H}^1(\Omega)} .
$$
Iterating this argument, we construct a solution $u^\eps$ on any interval
$[0,T]$ satisfying in particular, for all $k \in \N$ such that $kT_* \le
T$,
$$ 
||u^\eps||_{L^\infty (0,kT_*,\dot{H}^1(\Omega))}\leq \sqrt{2}^k R
||u_0||_{\dot{H}^1(\Omega)}.
$$

\paragraph{Energy and entropy estimates.}
The conservation of mass follows from Proposition~\ref{prop:corner},
but we need to explain how to derive \eqref{eq:nrj-eps},
\eqref{eq:estim-h1} and \eqref{eq:ent-eps} from \eqref{eq:reg-weak}.
Formally, one has to choose successively $\phi=
  -u^\eps_{xx}+I(u^\eps)$, $\phi=-u^\eps_{xx}$ and $\phi=G'_\eps
  (u^\eps)$. Making such a formal computation rigourous is quite
  standard;  details are given in
  Appendix~\ref{app:nrj-entropy} for the reader's convenience.
Finally, (\ref{eq:holder}) follows from (\ref{eq:holder-corner}).
\end{proof}

\begin{proof}[Proof of Lemma \ref{lem:pf}]
  We need to check that the conditions of Leray-Schauder's fixed point
  theorem are satisfied:

\paragraph{$\F$ is compact.}
Let $(v_n)_n$ be a bounded sequence in $V$ and let $u_n$ denote
$\F(v_n)$.  The sequence $(I (v_n))_x$ is bounded in $L^2 (Q)$, and so
$$
g_n = f_\eps (v_n) \partial_x I (v_n) \text{ is bounded in } L^2 (Q).
$$
Estimate~\eqref{eq:nrj-corner} implies that $u_n$ is bounded in
$L^2(0,T_*;H^3_N(\Omega))$.  In particular $\pa_{xxx} u_n$ is bounded
in $L^2 (0,T_*;L^2(\Omega))$ and Equation \eqref{eq:fp} implies that
$\pa_t(u_n)$ is bounded in $L^2(0,T_*,H^{-1}(\Omega))$.  Using Aubin's
lemma, we deduce that $(u_n)_n$ is pre-compact in
$V=L^2(0,T_*;H^2_N(\Omega))$.

\paragraph{$\F$ is continuous.}
Consider now a sequence $(v_n)_n$ in $V$ such that $v_n \to v$ in $V$
and let $u_n=\F(v_n)$.  We have in particular $v_n \to v$ in $L^2 (Q)$
and, up to a subsequence, we can assume that $v_n \to v$ almost
everywhere in $Q$.  Hence, $f_\eps (v_n) \to f_\eps (v)$ almost
everywhere in $Q$.  We also have that $(I (v_n))_x $ converges to $(I
(v))_x$ in $L^2 (Q)$, and since $|f_\eps(v_n)| \leq M$ a.e., we can
show that
$$
g_n = f_\eps (v_n) \partial_x I (v_n) \to f_\eps (v) \partial_x I (v)
= g \quad \text{ in } L^2 (Q).
$$

Next, the compacity of $\F$ implies that $(u_n)_n$ is pre-compact in
the space $L^2(0,T;H^2_N(\Omega))$, and so $u_n$ converges (up to a
subsequence) to $U$ in $V$. In particular, $u_n \to U$ in $L^2 (Q)$
and (up to a another subsequence), $u_n \to U$ almost everywhere in
$Q$.  We thus have $f_\eps (u_n) \to f_\eps (U)$ in $L^2(Q)$, and
passing to the limit in the equation, we conclude that $U= u=\F(v)$
(by the uniqueness result in Proposition~\ref{prop:corner}).  Since
this holds for any subsequence of $u_n$, we deduce that the whole
sequence $u_n$ converges to $u$ hence
$$\F(v_n) \to \F(v) \mbox{  in $V$ as $n \to \infty$}$$
and $\F$ is continuous.

\paragraph{A priori estimates.}
It only remains to show that there exists a constant $R>0$ such that
for all functions $u \in V$ and $\sigma \in [0,1]$ such that $u =
\sigma \F(u)$, we have
$$
\|u \|_V \le R.
$$ 
This is where the smallness of $T_*$ will be needed.

Using \eqref{eq:nrj-corner}, we see that 
\begin{multline}\label{estim1}
 \intom (u_x(t))^2 \, dx + \frac\eps2 \int_0^{T_*} \intom  (u_{xxx})^2  \,
dx\,dt \\ \le  \frac{M}2 \int_0^{T_*} \int_\Omega((I(u)_x)^2 \, dx\,dt +  \intom (u_0)_x^2 \, dx,
\end{multline}
and using \eqref{eq:semi-norm}
and the interpolation inequality
\begin{equation}\label{eq:interpol}
\|u_x\|_{L^2(\Omega)} \le C \|u\|_{L^2 (\Omega)}^{\frac12} \| u_{xx}\|_{L^2(\Omega)}^{\frac12},
\end{equation}
 we get:
\begin{eqnarray}
  \nonumber  \int_0^{T_*}\int_\Omega ((I(u)_x)^2  \, dx \, dt & \le &  
  2C_\eps \int_0^{T_*} \|I(u)(t)\|^2_2 \, dt + \frac{\eps}{2M} \int_0^{T_*}
  \|(I(u))_{xx}(t)\|^2_2 \, dt \\
\nonumber   & \le &  
  2C_\eps \int_0^{T_*}  \|u_x(t)\|^2_2 \, dt + \frac{\eps}{2M} \int_0^{T_*}
  \|u_{xxx}(t)\|^2_2 \, dt \\
  \label{estim2}  & \le & 2C_\eps {T_*} \hspace{-3pt} \sup_{t \in [0,{T_*}]} \|u_x(t)\|^2_2 
  +  \frac{\eps}{2M} \int_0^{T_*} \hspace{-1pt} \|u_{xxx}(t)\|^2_2 \, dt
\end{eqnarray}
where $C_\eps$ only depends on the constant $C$ in~\eqref{eq:interpol} and the parameters $\eps$
and $M$. Combining \eqref{estim1} and \eqref{estim2}, we conclude that
$$
(1 - MC_\eps {T_*}) \sup_{t \in [0,T]} \|u_x(t)\|^2_2 + \frac\eps4 \int_0^{T_*}
  \|u_{xxx}(t)\|^2_2 \, dt \le  \intom (u_0)_x^2 \, dx.
$$
Therefore, choosing $ T_*:= \frac1{2MC_\eps}$, we get the following estimates
$$
\|u\|_{L^\infty(0,{T_*};\dot{H}^1(\Omega))} \le \sqrt{2} || u_0||_{\dot{H}^1(\Omega)}\quad \mbox { and }
\|u\|_{L^2(0,{T_*};\dot{H}^3_N(\Omega))} \le \frac{2}{\sqrt{\eps}} || u_0||_{\dot{H}^1(\Omega)}
$$
Since we also have 
$$\intom u(t) \, dx = \intom u_0 \, dx,$$
we deduce that $\|u \|_V \le R$ for some constant $R$ depending on $\eps$,
which completes the proof.
\end{proof}

\section{Proof of Theorem~\ref{thm:main}}
\label{sec:proof1}

As pointed out in the introduction, one of the main difficulties in the
proof of Theorem~\ref{thm:main} is that the natural energy estimate
(\ref{eq:nrj-corner}) does not give any information by itself, since
$\E(u)$ may be negative. Even if Lemma \ref{lem:h1} implies that the
quantity
$$
\alpha \E(u^\eps)+\beta||u^\eps||_{L^1(\Omega)} 
$$
is bounded below by the $H^1$ norm of $u$, the mass conservation only
allows us to control the $L^1$ norm of $u^\eps$ if we know that
$u^\eps$ is non-negative.  Unfortunately, it is well known that
equation (\ref{eq:reg}) does not satisfy the maximum principle, and
that the existence of non-negative solutions of (\ref{eq:0}) is
precisely a consequence of the degeneracy of the diffusion
coefficient, so that while we can hope (and we will prove) to have
$\lim_{\eps\to0}u^\eps\geq 0$ we do not have, in general, that
$u^\eps\geq 0$.  Lemma \ref{lem:add} below will show that it is
nevertheless possible to derive some a priori estimates that are
enough to pass to the limit, provided the initial entropy is finite
(\ref{eq:entropy_init}).

\begin{proof}[Proof of Theorem~\ref{thm:main}]
  Consider the solution $u^\eps$ of (\ref{eq:reg}) given by
  Theorem~\ref{thm:existence-reg}.  In order to prove Theorem
  \ref{thm:main}, we need to show that $\lim_{\eps\to0}u^\eps$ exists and
  solves (\ref{eq:weak}).

Since we cannot use the energy inequality to get the necessary
estimates on $u^\eps$, we will need  the following lemma:
\begin{lemma}\label{lem:add}
Let $H_\eps$ denote the following functional:
$$
H_\eps (v) = \intom [v_x^2  + 2 M G_\eps (v) ] \, dx.
$$
Then the solution $u^\eps$ given by Theorem~\ref{thm:existence-reg} satisfies
$$
H_\eps (u^\eps(t)) + \frac M 2 \int_0^t \intom (u^\eps_{xx})^2 \, dx\, ds +  \frac12 \int_0^t \intom
f_\eps(u^\eps)(u_{xxx}^\eps)^2 \, dx\, ds  \le H_\eps (u_0) e^{t/2}
$$
for every $t \in [0,T]$.
\end{lemma}
\begin{proof}[Proof of Lemma \ref{lem:add}]
Using \eqref{eq:estim-h1} and \eqref{eq:ent-eps} we see that
\begin{multline*}
 \intom G_\eps(u^\eps (t))\, dx+ \frac12 \int_0^t \intom (u^\eps_{xx})^2 \, dx\, ds \\
\le \intom G_\eps(u_0) \, dx + \frac12 \int_0^t \intom (u^\eps_x)^2
 \, dx\, ds
\end{multline*} 
and   
\begin{multline*}
\intom (u^\eps_x)^2  \, dx + \frac12 \int_0^t \intom
f_\eps(u^\eps)[u_{xxx}^\eps]^2  \, dx\, ds \\
\le \intom((u_0)_x)^2 \, dx +\frac{M}2 \int_0^t \intom (u^\eps_{xx})^2 \, dx\, ds.
\end{multline*}
This implies
$$
H_\eps (u^\eps (t)) \le H_\eps (u_0) + \int_0^t H_\eps (u^\eps(s))
ds,
$$
and Gronwall's lemma yields the desired result.
\end{proof}

\paragraph{Sobolev and H\"older bounds.}
We now gather all the a priori estimates: Using the conservation of mass,
Lemma \ref{lem:add} and inequality \eqref{eq:entropy_init}, we see that there exists a
constant $C$ independent of $\eps$ such that:
\begin{eqnarray}
\sup_{t\in[0,T]}\intom (u^\eps_x(t))^2 \, dx   & \leq &  C, \label{eq:H1est} \\
\sup_{t\in[0,T]}\intom G_\eps(u^\eps (t))\, dx & \leq & C, \label{eq:entrest}\\
 \int_0^T \intom
f_\eps(u^\eps)[u_{xxx}^\eps]^2  \, dx\, dt & \leq  & C,\label{eq:xxxest}\\
\int_0^T \intom (u^\eps_{xx})^2 \, dx \, dt & \leq & C. \label{eq:xxest}
\end{eqnarray}
Next, we note that \eqref{eq:H1est} yields
$$ 
\E(u^\eps) \geq - ||u^\eps||_{L^\infty(0,T;\dot{H}^{1/2}(\Omega))} \geq -
C ||u^\eps||_{L^\infty(0,T;H^{1}(\Omega))}  \geq -C
$$
and so \eqref{eq:nrj-eps} gives
\begin{equation}\label{eq:dissest}
\int_0^T \intom f_\eps (u^\eps) \big[u^\eps_{xxx} -
I(u^\eps)_{xx}\big]^2 \, ds\, dx \leq C.
\end{equation}
Finally, estimates~\eqref{eq:holder}, \eqref{eq:H1est} and
  \eqref{eq:dissest} yield that $u^\eps$ is bounded in $\mathcal
  C^{1/2,1/8}_{x,t}(Q)$.

\paragraph{Limit $\mathbf{\eps\to 0}$.}
The previous H\"older estimate implies that there exists a function
$u(x,t)$ such that $u^\eps$ converges uniformly to $u$ as $\eps$ goes
to zero (up to a subsequence).
 Inequality (\ref{eq:xxest}) also
implies that
$$ u^\eps\rightharpoonup u \mbox{ in }  L^2(0,T;H^2(\Omega)) \mbox{-weak}
$$
and Aubin's lemma gives
$$  u^\eps \longrightarrow u \mbox{ in }  L^2(0,T;H^1(\Omega))\mbox{-strong}.
$$
After integration by parts, (\ref{eq:reg-weak}) can be written as
\begin{multline*}
  \iint_Q u^\eps  \phi_t  - f_\eps (u^\eps) [ u^\eps_{xx}-
  I(u^\eps)] \phi_{xx}  - f_\eps'(u^\eps)u^\eps_x (u^\eps_{xx}-I(u^\eps)) \phi_x  \, dt\, dx  
 \\ = \intom u_0(x)\phi(0,x) \, dx
\end{multline*}
and passing to the limit $\eps\to 0$ gives (\ref{eq:weak}).

\paragraph{Non-negative solution.}
It only remains to show that $u$ is non-negative. This can be done as in
\cite{bf90}, using (\ref{eq:entrest}) (and the fact that $f$ satisfies (\ref{eq:f}) with $n>1$).

\paragraph{$L^\infty$ a priori estimate.}
Finally, (\ref{eq:H1est}) and Sobolev's embedding implies
that there exits a constant $M_0$ depending only on $
||u_0||_{H^1(\Omega)}$ such that
\begin{equation}\label{eq:linfty}
 0\leq u(t,x) \leq M_0.
\end{equation}
Choosing $M> M_0$, we deduce that $u$ solves (\ref{eq:0}).
\end{proof}

\section{Proof of Theorem \ref{thm:second}}
\label{sec:proof2}

In order to get Theorem~\ref{thm:second}, we need to derive the following
corollary from Theorem~\ref{thm:main}.
%-----------------
\begin{corollary}\label{cor:thm-main}
  The solution $u$ constructed in Theorem~\ref{thm:main} satisfies for all
  $\phi \in \dis ((0,T) \times \bar \Omega)$,
\begin{equation}\label{eq:weak-P}
  \iint_Q u  \phi_t \, dt\, dx + \iint_P f (u) [ u_{xx}-
  I(u)]_x \phi_x  \, dt\, dx = 0
\end{equation}
where $P = \{ (x,t) \in \bar Q : u(x,t) >0,t>0\}$. 
\end{corollary}
%---------------
\begin{proof}
  In view of the proof of Theorem~\ref{thm:main}, $u$
    is the uniform limit of a subsequence of $(u^\eps)_{\eps >0}$
    where $u^\eps$ is given by Theorem~\ref{thm:existence-reg}. Since
    $u^\eps$ satisfies \eqref{eq:reg-weak}, it is thus enough to pass
    to the limit in this weak formulation as $\eps \to 0$ in order to
    get the desired result.  Let $h^\eps$ denote $f_\eps
    (u^\eps)[u^\eps_{xx}-I(u)]_x$. Estimates~\eqref{eq:dissest} and
    \eqref{eq:linfty} imply
\begin{equation}
\label{estim:heps}
\iint_Q h_\eps^2 \, dx \, dt \le C. 
\end{equation}
In other words, $(h^\eps)_\eps$ is bounded in $L^2 (Q)$.  Hence, up to
a subsequence, 
$$ 
h^\eps \rightharpoonup h \qquad \mbox{ in $  L^2(Q)-$weak}.
$$
Furthermore, we recall that  there exists a continuous function $u(x,t)$ such that $u^\eps$ converges uniformly
to $u$ as $\eps$ goes to zero (up to a subsequence).

Passing to the limit in \eqref{eq:reg-weak}, we deduce that the
function $u$ satisfies
$$
\iint_Q u  \phi_t \, dt\, dx + \iint_\Omega h \phi_x  \, dt\, dx =0.
$$
We now have to show that
$$ 
h= \left\{
\begin{array}{ll}
0 &\mbox{ in }\{u=0\},\\
f (u) \big[u_{xx} -I(u)\big]_x&\mbox{ in }P=\{u>0\}.
\end{array}
\right.
$$
First we note that for any test function $\phi$ and $\eta>0$, we have
\begin{multline*}
  \left| \int_0^T\int_{\{u\leq \eta \}}
    f_\eps (u^\eps) \big[u^\eps_{xxx} -I(u^\eps)_{x}\big] \phi \, dx\, dt\right|\\
  \leq C (\phi) \bigg(f_\eps
  (3\eta/2) \bigg)^{1/2}\left( \int_0^T\int_{\{u\leq
      \eta\}} f_\eps (u^\eps) \big[u_{xxx} -I(u)_{x}\big]^2 \, dx\,
    dt\right)^{1/2}
\end{multline*}
for $\eps$ small enough (so that $|u^\eps-u|\leq \eta /2$).
Inequality (\ref{eq:dissest}) thus implies
$$ 
\limsup_{\eps\to 0} \left| \int_0^T\int_{\{u\leq 2 \eta\}} f_\eps
  (u^\eps) \big[u^\eps_{xxx} -I(u^\eps)_{x}\big] \phi \, dx\, dt
\right| \leq C(\phi) f(\eta/2)^{1/2}.
$$
We deduce (since $f(0)=0$)
\begin{equation}\label{eq:h=0}
h=0 \qquad \mbox{ on } \{u=0\}.
\end{equation}
Next,  \eqref{estim:heps} yields (for $\eps>0$ and $\eta$ small enough)
$$ 
 \iint_{\{u > 2\eta \}} |u_{xxx}^\eps-I(u^\eps)_x|^2  \, dx\, ds \leq C(\eta).
$$
This implies that, if $Q_\eta$ denotes $\{ u > 2 \eta\}$, $(u_{xxx}^\eps
- I (u^\eps)_x)$ is bounded in the space $L^2 (Q_\eta)$. Hence, we can
extract from $(u_{xxx}^\eps - I (u^\eps)_x)_{\eps>0}$ a subsequence
converging weakly in $L^2 (Q_\eta)$. Moreover, remark that $Q_\eta$ is an
open subset of $Q$ (recall that $u$ is H\"older continuous) and
$u^\eps_{xxx} -I(u^\eps)_x$ converges in the sense of distributions to
$u_{xxx}- I (u)_x$ (use the integral representation for $I(\cdot)$). We thus conclude that,
$$
u^\eps_{xxx}-I(u^\eps)_x \rightharpoonup u_{xxx}-I(u)_x \quad
\text{ in } L^2 (Q_\eta).
$$
This yields
$$
h=  f (u) \big[u_{xxx} -I(u)_{xx}\big]\quad  \mbox{ in } \{u>0\}
$$
which concludes the proof of Corollary~\ref{cor:thm-main}.
\end{proof}

We now turn to the proof of Theorem~\ref{thm:second}. 
\begin{proof}[Proof of Theorem \ref{thm:second}]
When $u_0$ does not satisfy (\ref{eq:entropy_init}), we lose the
$L^2(0,T;H^2(\Omega))$ bound on $u^\eps$, and the previous analysis fails.
However, we can introduce
$$
u_0^\delta=u_0+\delta
$$
which satisfies (\ref{eq:entropy_init}).  Theorem \ref{thm:main} then
provides the existence of a {\it non-negative} solution $u^\delta$ of
\eqref{eq:weak}. In view of Corollary~\ref{cor:thm-main}, $u^\delta$
satisfies:
\begin{equation}\label{eq:eq_delta}
\iint_Q u^\delta  \phi_t \, dt\, dx + \iint_P f (u^\delta) [ u^\delta_{xx}-
I(u^\delta)]_x \phi_x  \, dt\, dx = 0.
\end{equation}
Since $u^\delta$ is non-negative, the conservation of mass gives a bound in
the space $L^\infty(0,T;L^1(\Omega))$ and allows us to make use of the
energy inequality: Indeed, using~(\ref{eq:nrj}) and Lemma~\ref{lem:h1} we
see that there exists a constant $C$ independent of $\delta$ such that
$$ ||u^\delta||_{L^\infty(0,T;H^1(\Omega))} \leq C$$
and
\begin{equation}
\int_0^T \intom f (u^\delta) \big[u^\delta_{xxx} -
I(u^\delta)_{x}\big]^2 \, ds\, dx  \leq   C. \label{eq:dissest_delta}
\end{equation}
We now define the flux
$$ 
h^\delta= f (u^\delta) \big[u^\delta_{xxx} -I(u^\delta)_{x}\big].
$$
Inequality (\ref{eq:dissest_delta}) implies that $h^\delta$ is bounded in
$L^2(Q)$, and that there exists a function $h\in L^2(Q)$ such that
$$ h^\delta \rightharpoonup h \qquad \mbox{ in $  L^2(Q)-$weak}.$$
Proceeding as in the proof of Theorem \ref{thm:main}, we deduce that
$u^\delta$ is bounded in $\mathcal C^{1/2,1/8}(\Omega\times(0,T))$ and
that there exists a function $u(x,t)$ such that $u^\delta$ converges
uniformly to $u$ as $\delta$ goes to zero (up to a
subsequence). We can now argue (with minor changes) as
  in the proof of Corollary~\ref{cor:thm-main} and conclude.
\end{proof}

\appendix

\section{Proof of Proposition~\ref{prop:corner}}
\label{app:corner-proof}
Our goal here is to prove the existence of a  weak solution of
$$ 
u_t + (a u_{xxx})_x = (ag)_x.
$$
We first prove the following proposition.
%--------------------------------------------------------------------
\begin{proposition} \label{prop:stat} For all $h \in H^1 (\Omega)$,
  there exists $v \in V_0:=H^1 \cap H^3_N$ such that for all $\phi \in
  \dis (\bar \Omega)$,
\begin{equation}\label{eq:weak-stat}
- \intom \frac{v-h}\tau \phi \, dx + \int a v_{xxx} \phi_x \,dx=
\intom a g \phi_x \, dx.
\end{equation}
In particular,
\begin{eqnarray}
\intom v \, dx &=& \intom h \, dx, \nonumber \\
\frac12 \intom v_x^2 + \tau \intom a v_{xxx}^2& \le & \frac12 \intom
h_x^2 + \tau \intom a g v_{xxx}. \label{estim-stat}
\end{eqnarray}
\end{proposition}
%----------------
\begin{proof}
In order to prove this proposition, we have to reformulate the
equation. More precisely, instead of choosing test functions $\phi
\in \dis (\bar \Omega)$, we choose $\phi = - \psi_{xx} + \intom
\psi \, dx$ where $\psi$ is given by the following lemma
%------------------------------
\begin{lemma}\label{lem:reform}
  For all $\phi \in \dis(\bar \Omega)$, there exists $\psi
  \in \dis (\bar \Omega)$ such that
$$
- \psi_{xx} + \intom \psi \, dx = \phi.
$$ 
\end{lemma}
%-----------
Hence, we consider $V_0 := H^1  \cap H^3_N$ equipped with the
norm $\|v\|^2_{V_0} = \|v_{xxx}\|^2_{L^2}+(\int v\, dx)^2$ and we look for $v \in V_0$ such
that for all $\psi \in \dis (\bar \Omega)$,
\begin{multline} \label{eq:new-form}
    \intom v_x \psi_{x} \, dx + \tau \int a v_{xxx} \psi_{xxx} \,dx 
+ \left(\intom v \, dx \right) \left(\intom \psi \, dx\right)\\
  = \intom h_x \psi_{x} \, dx + \left(\intom h \, dx
  \right) \left(\intom \psi \, dx\right) + \tau \intom a g \psi_{xxx}
  \, dx.
\end{multline}
We thus consider the bilinear form $A$ in $V_0$ defined as follows:
for all $v,w \in V_0$, 
$$
A (v,w)= \intom v_x w_x \, dx +  \tau \int a v_{xxx} w_{xxx} \,dx 
+ \left(\intom v \, dx \right) \left(\intom w \, dx\right). 
$$
We check that it is continuous and coercive:
\begin{eqnarray*}
|A(v,w)| &\le& \|v_x\|_2 \|w_x\|_2 + M\tau \|v_{xxx}\|_2
\|w_{xxx}\|_2 + \|v\|_1 \|w\|_1  \\
&\le& C \|v\|_{V_0} \|w\|_{V_0}\\
A(v,v) &\ge& \intom [ (v_x)^2 + \eps \tau (v_{xxx})^2 ]\, dx +\left( \intom v\, dx\right)^2 \ge \|v\|^2_{V_0}.
\end{eqnarray*}
We now consider the following linear form $L$ in $V_0$: for all $w \in
V_0$,
$$
L(w) = \intom h_x w_{x} \, dx + \left(\intom h \, dx \right)
\left(\intom w \, dx\right) + \tau \intom a g w_{xxx} \, dx
$$
Since $0 \le a \le M$ and $g \in L^2$, $L$ is continuous as soon as $h
\in H^1(\Omega)$. Lax-Milgram theorem thus implies that there exists
$v \in V_0$ such that \eqref{eq:new-form} holds true for all $w \in
V_0$. 

Eventually, remark that conservation of mass and
  \eqref{estim-stat} are direct consequences of
  \eqref{eq:weak-stat}.  The proof of Proposition~\ref{prop:stat} is
now complete.
\end{proof}
We can now prove Proposition~\ref{prop:corner}. 
\begin{proof}[Proof of Proposition~\ref{prop:corner}]
  For any $\tau >0$, we consider $N_\tau = \lceil
    \frac{T}{\tau} \rceil$.  We then define inductively a sequence
    $(u^n)_{n=0,\dots,N_\tau}$ of $V_0$ as follows: $u^0 =u_0$ and
    $u^{n+1}$ is obtained by applying Proposition~\ref{prop:stat} to
    $h = u^n$. We then define $u^\tau:[0,N_\tau \tau) \times \Omega$ as
    follows:
$$
u^\tau (t,x) = u^n (x) \quad \text{ for } t \in [n\tau, (n+1)\tau).
$$
We have $\intom u^\tau(t,x) \, dx = \intom u_0 (x) \, dx$ for all
$t$. We also derive from \eqref{estim-stat} that we have
\begin{multline*}
\intom (u^\tau_x)^2 (T,x) \, dx + \int_0^T \intom a (u^\tau_{xxx})^2
(t,x) \, dt dx \\ \le \intom ((u_0)_x)^2 \, dx + \int_0^T \intom (a g
u_{xxx})(t,x) \, dt dx
\end{multline*}
In particular, $(u^\tau)_\tau$ is bounded in $L^\infty (0,T;H^1
(\Omega))$ and $(S_\tau u^\tau -u^\tau)_\tau$ is bounded in $L^2
(0,T-\tau;H^{-1} (\Omega))$ where $S_\tau v (t,x)= v(t+\tau,x)$. We
derive from \cite[Theorem~5]{simon} that $(u^\tau)_\tau$ is relatively compact in
$\mathcal{C} (0,T; L^2 (\Omega))$.

We now have to pass to the limit in
  \eqref{eq:weak-stat}. Since $(u^\tau_{xxx})_\tau$ is bounded in $L^2
(Q)$ and we can find a sequence $\tau_n \to 0$ such that $u^{\tau_n}
\to u$ in $\mathcal{C}(0,T,L^2(\Omega))$ and $u^{\tau_n}_{xxx} \to
u_{xxx}$ in $L^2 (Q)$. This is enough to conclude.

We next explain how to get
  \eqref{eq:holder-corner}. Sobolev's embedding imply that there
  exists a constant $K$ (depending on
  $\|u\|_{L^\infty(0,T;H^1(\Omega))}$) such that
$$
| u(x_1,t)-u(x_2,t)| \leq K|x_1-x_2|^{1/2}
$$
for all $x_1$, $x_2\in \Omega$ and for all $t\in (0,T)$.
Since $u$ satisfies
$$ 
u_t =h_x
$$
with $h \in L^2(Q)$, 
it is a fairly classical result that H\"older regularity in space
implies H\"older regularity in time. More precisely, we have (see
\cite{bf90}, Lemma 2.1 for details):
%---------
\begin{lemma}\label{lem:holder}
  There exists a constant $C$ such that for all $x_1$, $x_2$ in
  $\Omega $ and all $t_1$, $t_2>0$,
$$ 
|u(x_1,t_1)-u(x_2,t_2)|\leq C|x_1-x_2|^{1/2}+C|t_1-t_2|^{1/8}.
$$
\end{lemma}
The proof of Proposition~\ref{prop:corner} is thus  complete.
 %---------
\end{proof}

\section{Proof of \eqref{eq:nrj-eps}, \eqref{eq:estim-h1} and \eqref{eq:ent-eps}}
\label{app:nrj-entropy}

We have to derive \eqref{eq:nrj-eps}, \eqref{eq:estim-h1} and
\eqref{eq:ent-eps} from \eqref{eq:reg-weak}. Using the fact that $u^\eps$
lies in $\mathcal{C}([0,T], H^1 (\Omega))$, we first state the following
lemma
\begin{lemma}\label{lem:1}
For all $\phi \in \dis (\bar Q)$, 
\begin{multline}\label{eq:reg-weak-extended}
  \iint_Q u^\eps \phi_t + f_\eps (u^\eps) [ u^\eps_{xx}- I(u^\eps)]_x
  \phi_x \, dt\, dx \\= \intom u^\eps (x,T) \phi(x,T) - \intom u_0 (x) \phi
  (x,0) \, dx.
\end{multline}
\end{lemma}
The proof of such a lemma is fairly classical. It relies on mollifiers that
are decentered in the time variable. More precisely, one considers a smooth
even function $\rho:\R \to [0,1]$ compactly supported in $[-1,1]$ and such
that $\int \rho =1$. Then for $\alpha >0$ and $\delta \in \R$, one can
define 
$$
\rho_{\alpha,\delta} (t) = \tau_\delta \rho_\alpha (t) = \alpha \rho \left(
  \frac{t-\delta}{\alpha} \right).
$$
If now a function $f$ is defined in $[0,T]$, it can be extended by $0$ to
$\R$; in other words, it can be replaced with $f \un_\Omega$ where
$\un_\Omega (x)=1$ if $x \in \Omega$ and $\un_\Omega (x)=0$ if not; then
the convolution product in $\R$: $(f \un_\Omega) \star
\rho_{\alpha,\delta}$ is a smooth function in $\R$ which vanishes near
$t=0$ (resp. $t=T$) if $\delta >0$ (resp. $\delta <0$).

Consider a smooth function $\rho$ such as in the proof of
Lemma~\ref{lem:1}. Consider $\alpha >0$ and define $\rho_\alpha (x) =
\alpha \rho (\frac{x}\alpha)$. Then consider $\theta
(x,t)=\rho_\alpha (x) \rho_\alpha (t)$. 
%-------------
\begin{lemma}\label{lem:2}
  Recall that $u^\eps \in L^\infty (Q)$ and $h^\eps = f_\eps (u^\eps) [
  u^\eps_{xx}- I(u^\eps)]_x \in L^2 (Q)$. Then for all  $v \in L^1 (Q)$,
\begin{equation}
\label{eq:1}  \iint_Q u^\eps (v \star \theta)_t = \iint_Q
  (u^\eps \star \theta)_t v
\end{equation}
where $f \star g$ means $ (f \un_Q) \star (g \un_Q)$. 
\end{lemma}
%----------
We next apply Lemma~\ref{lem:1} with $\phi = v\star \theta$ where $v$ is
chosen to be successively $u^\eps_{xx} \star \theta$, $I(u^\eps \star
\theta)$ and $G'_\eps (u^\eps \star \theta)$. After direct computations, we
can let $\alpha \to 0$ and get the desired estimates.

\def\cprime{$'$}

\end{document}